\documentclass[12pt]{amsart}
\usepackage{amssymb}
\usepackage{amscd}

\newtheorem{theorem}{Theorem}[section]
\newtheorem{lemma}[theorem]{Lemma}

\newtheorem{corollary}[theorem]{Corollary}

\theoremstyle{definition}
\newtheorem{definition}[theorem]{Definition}

\theoremstyle{remark}
\newtheorem{remark}[theorem]{Remark}
\numberwithin{equation}{section}

\newcommand{\C}{\mathbb{C}}

\newcommand{\N}{\mathbb{N}}

\newcommand{\R}{\mathbb{R}}
\newcommand{\T}{\mathbb{T}}
\newcommand{\Z}{\mathbb{Z}}

\newcommand\cA{\mathcal{A}}

\newcommand\cE{\mathcal{E}}
\newcommand\cF{\mathcal{F}}
\newcommand\cH{\mathcal{H}}

\newcommand\cO{\mathcal{O}}
\newcommand\cP{\mathcal{P}}
\newcommand\cQ{\mathcal{Q}}

\newcommand\cU{\mathcal{U}}
\newcommand\fA{\mathfrak{A}}

\newcommand\Ad{\operatorname{Ad}}

\newcommand\Aut{\operatorname{Aut}}

\newcommand\id{\operatorname{id}}

\newcommand\ind{\operatorname{ind}}
\newcommand\tr{\operatorname{tr}}
\newcommand\inpr[2]{\langle{#1,#2}\rangle}
\newcommand\tG{\tilde{G}}

\begin{document}
\title[$E_0$-semigroups]{$E_0$-semigroups: around and beyond Arveson's work}

\dedicatory{Dedicated to the memory of Bill Arveson}

\author[Masaki Izumi]{Masaki Izumi}

\address{MASAKI IZUMI, Department of Mathematics\\ Graduate School of Science\\
Kyoto University\\ Sakyo-ku, Kyoto 606-8502\\ Japan}
\email{izumi@math.kyoto-u.ac.jp}

\begin{abstract} 
We give an account of the theory of $E_0$-semigroups.
We first focus on Arveson's contributions to the field and related results. 
Then we present the recent development of type II and type III $E_0$-semigroups. 
We also include a short note in Appendix, based on Arveson's observation, 
on noncommutative Poisson boundaries.
\end{abstract}

\begin{subjclass}{Primary 46L55; Secondary 46L57, 46L53.}
\end{subjclass}

\maketitle

\section{Introduction} 
Arveson opened up new frontiers in several areas in operator algebras, 
or more generally noncommutative analysis. 
He often provided essential ideas at the very beginning of new subjects, and then generously left further development 
to others, as probably seen from Kenneth Davidson's article, also in this volume. 
One specific example is the theory of $E_0$-semigroups, the main subject of this article, 
and the author is one of Arveson's followers who has luckily obtained substantial benefit from his ideas. 
These ideas were summarized as a monograph \cite{Arv2003-1} by himself, which is the standard reference for 
$E_0$-semigroups. 
Two main subjects of the monograph, among others, are product systems 
(infinite tensor product systems of Hilbert spaces) and the dilation theory of 
semigroups of completely positive maps, which reflects the importance of these subjects in the field well.  
Therefore it is natural for us to discuss them in this article too.

\begin{definition}
An $E_0$-semigroup $\alpha=\{\alpha_t\}_{t\geq 0}$ is a family of unital normal $*$-endomorphisms of 
a von Neumann algebra $M$ satisfying the following conditions: 
\begin{itemize}
\item[(i)] $\alpha_0=\id$, 
\item[(ii)] $\alpha_s\circ \alpha_t=\alpha_{s+t}$ for all $s,t\geq 0$, 
\item[(iii)] For every $a\in M$ and every normal functional $\varphi\in M_*$, 
the function $[0,\infty)\ni t\mapsto \varphi(\alpha_t(a))\in \C$ is continuous. 
\end{itemize}
\end{definition}

Throughout the article, we concentrate on $E_0$-semigroups acting on $B(H)$, the set of bounded operators 
of a separable infinite dimensional complex Hilbert space $H$. 
In this particular case the normality condition of $\alpha_t$ is redundant (see \cite[Chapter V, Theorem 5.1]{Ta}). 
Since $B(H)$ is believed to be the easiest infinite dimensional factor, one may wonder if something significant 
comes out of such objects. 
Yet, it turned out that $E_0$-semigroups on $B(H)$ are already rich and difficult, and they have deep connections with 
various areas in mathematics such as probability theory and classical analysis 
(not to mention functional analysis). 

Soon after Powers \cite{P88} initiated the systematic analysis of $E_0$-semigroups in the late 80's, Arveson \cite{Arv89-1} 
introduced the notion of a product system, which is the key concept for later development of the theory. 
In particular, it inspired the probabilistic approach of Tsirelson \cite{T2003-1}. 
Arveson \cite{Arv89-2},\cite{Arv90-2} showed that there is a one-to-one correspondence between 
the cocycle conjugacy classes of $E_0$-semigroups and the isomorphism classes of product systems, 
which reduced the classification problem of the former to that of the latter. 
Then he completely classified so-called type I $E_0$-semigroups by an invariant called index, 
which is the first substantial classification result in the theory. 
There are a lot of other important invariants for $E_0$-semigroups defined via the corresponding product systems now. 

In general, it is not so easy to construct new examples of $E_0$-semigroups. 
However, it is relatively easy to construct semigroups of unital completely positive (ucp) maps, and 
they give rise to $E_0$-semigroups through Bhat's dilation theorem \cite{B96}. 
This approach was extensively developed by Arveson \cite{Arv99} and Powers \cite{P2003-1}, \cite{P2003-2}, 
and it is one of the most important constructions of $E_0$-semigroups. 

For the proofs and related results of the statements we present in this article, the reader is referred to Arveson's monograph 
\cite{Arv2003-1}. 
Another standard reference of $E_0$-semigroups is the conference proceedings \cite{Pr} with several valuable articles, 
including Arveson's survey \cite{Arv2003-2}, which appeared at the right time when the field sufficiently matured. 
\cite{Arv2003-2} can be read as an introduction to \cite{Arv2003-1}.

This article ends with a short note in Appendix on noncommutative Poisson boundaries, a notion introduced by the author for 
normal ucp maps acting on von Neumann algebras. 
At the occasion of a workshop taking place at the Fields Institute in 2007, Arveson briefly mentioned to the author that 
the noncommutative Poisson boundary of a normal ucp map is identified with the fixed point algebra of the minimal dilation of 
the given normal ucp map. 
This observation has potentially useful consequences, and we include them.  

\section{Around Arveson's work}

\subsection{Basic equivalence relations}
We first introduce two basic equivalence relations of $E_0$-semigroups, conjugacy and cocycle conjugacy. 
One of the main goals in the theory of $E_0$-semigroups is to classify them up to cocycle conjugacy. 

\begin{definition} Let $\alpha$ and $\beta$ be $E_0$-semigroups acting on $B(H)$ and $B(K)$ respectively. 
\begin{itemize} 
\item[(i)] We say that $\alpha$ and $\beta$ are conjugate if there exists a unitary $V$ from $H$ onto $K$ satisfying 
$\Ad V \circ \alpha_t=\beta_t\circ \Ad V$ for all $t\geq 0$, where $\Ad V(A)=VAV^*$. 
\item[(ii)] We say that a weakly continuous family of unitaries $U=\{U_t\}_{t\geq 0}$ in $B(H)$ is an $\alpha$-cocycle if 
they satisfy the 1-cocycle relation $U_{s+t}=U_s\alpha_t(U_t)$ for all $s,t\geq 0$. 
When $U$ is an $\alpha$-cocycle, we have a new $E_0$-semigroup defined by $\alpha^U=\{\Ad U_t\circ \alpha_t\}_{t\geq 0}$, 
which we call the perturbation of $\alpha$ by $U$. 
\item[(iii)] We say that $\alpha$ and $\beta$ are cocycle conjugate if a cocycle perturbation of $\alpha$ 
is conjugate to $\beta$. 
\end{itemize}
\end{definition}

To see the difference of the two equivalence relations above, we start with the trivial class, namely the class of  
$E_0$-semigroups $\alpha$ with $\alpha_t\in \Aut(B(H))$ for all $t\geq 0$. 
In this case, a classical result due to E. Wigner says that there exists a (possibly unbounded) self-adjoint operator $A$, 
uniquely determined up to additive constant, satisfying $\alpha_t=\Ad e^{itA}$. 
Therefore such $\alpha$ is always in the cocycle conjugacy class of $\id$, the semigroup consisting of only the identity map, 
while the classification of such $\alpha$ up to conjugacy is equivalent to the classification of $A$ up to additive constant 
and unitary equivalence. 
Therefore we see that the classification up to conjugacy is already complicated in the trivial case, and we cannot 
expect reasonable classification results for more general classes of $E_0$-semigroups. 

There is another reason why cocycle conjugacy is so natural. 
Let $\alpha$ be an $E_0$-semigroup acting on $B(H)$. 
The generator $\delta$ of $\alpha$ is defined by the limit
\begin{equation*}
\lim_{t\to+0}\frac{1}{t}(\alpha_t(A)-A),
\end{equation*}
in the strong operator topology, where the domain of $\delta$ is the set of $A\in B(H)$ for which the limit exists.  
Let $D\in B(H)$ be a self-adjoint operator, and let $\delta'(A)=\delta(A)+i[D,A]$.  
Then $\delta'$ generates an $E_0$-semigroup $e^{t\delta'}$ that is a cocycle perturbation of $\alpha$.  

\subsection{Product systems}
Before introducing a product system, we first discuss the Hilbert space consisting of the intertwining operators 
for a unital endomorphism $\rho$ of $B(H)$. 
Recall that $H$ is a separable infinite dimensional complex Hilbert space. 
The endomorphism $\rho$, regarded as a normal representation of $B(H)$ on $H$, 
is unitarily equivalent to a direct sum of copies of the identity representation of $B(H)$. 
This means that there exists an orthogonal decomposition $H=\oplus_{i\in I}H_i$ with 
$\rho(B(H))$-invariant closed subspaces $H_i\subset H$ such that the restriction of $\rho$ to $H_i$ is 
unitarily equivalent to the identity representation of $B(H)$. 
Thus there exist isometries $V_i\in B(H)$ whose range is $H_i$, satisfying 
$V_iA=\rho(A)V_i$ for all $A\in B(H)$, and in consequence 
\begin{equation}\label{endo}\rho(A)=\sum_{i\in I}V_iAV_i^*,\quad \forall A\in B(H),
\end{equation}
where the right-hand side converges in the strong operator topology.  

Let $\cH_\rho$ be the space of intertwining operators between $\id$ and $\rho$, that is 
\begin{equation*}
\cH_\rho=\{V\in B(H);\; VA=\rho(A)V,\;\forall A\in B(H)\}.
\end{equation*}
For $V,W\in \cH_\rho$, the product $W^*V$ is a scalar as it belongs to the center of $B(H)$. 
Equipped with inner product $\inpr{V}{W}_{\cH_\rho}1=W^*V$, the intertwiner space $\cH_\rho$ is a Hilbert space with an 
orthonormal basis $\{V_i\}_{i\in I}$. 
It is easy to see that the right-hand side of (\ref{endo}) does not depend on the choice of 
the orthonormal basis $\{V_i\}_{i\in I}$ of $\cH_\rho$. 
In fact for an arbitrary orthonormal system $\{V_i\}_{i\in I}$ of $\cH_\rho$, 
the equation (\ref{endo}) characterizes its completeness. 

Now we consider two unital endomorphisms $\rho,\sigma\in B(H)$. 
Then we have an inclusion relation $\cH_\rho\cdot \cH_\sigma\subset \cH_{\rho\circ \sigma}$, and for 
$V,V'\in \cH_\rho$ and $W,W'\in \cH_\sigma$,
\begin{align*}
\inpr{VW}{V'W'}_{\cH_{\rho\circ \sigma}} &={W'}^*{V'}^*VW=\inpr{V}{V'}_{\cH_\rho}{W'}^*W'
=\inpr{V}{V'}_{\cH_\rho}\inpr{W}{W'}_{\cH_\sigma} \\
 &=\inpr{V\otimes W}{V'\otimes W'}_{\cH_\rho\otimes \cH_\sigma}.
\end{align*}
Moreover if $\{V_i\}_{i\in I}$ and $\{W_j\}_{j\in J}$ are orthonormal bases of $\cH_\rho$ and $\cH_\sigma$ 
respectively, then $\{V_iW_j\}_{i\in I,j\in J}$ is an orthonormal basis of $\cH_{\rho\circ \sigma}$ as we have 
\begin{equation*}
\rho\circ \sigma(A)=\sum_{i\in I,j\in J}V_iW_jAW_j^*V_i^*.
\end{equation*}
This means that we can identify $\cH_{\rho\circ \sigma}$ with the Hilbert space tensor product 
$\cH_\rho\otimes \cH_\sigma$ under the identification of the product $VW$ in $B(H)$ and the simple tensor 
$V\otimes W$ in $\cH_\rho\otimes \cH_\sigma$. 

We get back to our original situation and consider an $E_0$-semigroup $\alpha$ acting on $B(H)$. 
Then we have a 1-parameter family of Hilbert spaces $\cE_\alpha(t):=\cH_{\alpha_t}$ for $t>0$, with identification 
$\cE_\alpha(s+t)=\cE_\alpha(s)\otimes \cE_\alpha(t)$, where the usual product in $B(H)$ corresponds to tensor product. 
Moreover the association $(0,\infty)\ni t\mapsto \cE_\alpha(t)$ should be measurable (or more strongly continuous) 
in an appropriate sense. 
Arveson \cite{Arv89-1} axiomatized this situation and introduced the notion of product systems. 

\begin{definition} A product system is a family of separable Hilbert spaces $p:E\rightarrow (0,\infty)$ over 
the half-line $(0,\infty)$, with fiber Hilbert spaces $E(t)=p^{-1}(t)$, endowed with a bilinear associative 
multiplication $E(s)\times E(t)\ni (x,y)\mapsto xy\in E(s+t)$ satisfying the following conditions:
\begin{itemize}
\item[(i)] $\inpr{xu}{yv}=\inpr{x}{y}\inpr{u}{v}$ for all $x,y\in E(s)$, $u,v\in E(t)$, 
\item[(ii)] The linear span of $E(s)\cdot E(t)$ is dense in $E(s+t)$, 
\item[(iii)] $E$ has the structure of a standard Borel space that is compatible with the projection 
$p:E\rightarrow (0,\infty)$, multiplication, the vector space operations, and the inner product.   
Moreover, there exist a separable infinite dimensional Hilbert space $H_0$ and a Borel isomorphism from $E$ onto $(0,\infty)\times H_0$ 
compatible with the projection $p$.
\end{itemize}
\end{definition}

\begin{remark}(1) 
The measurability condition (iii) is equivalent to the following condition: 
there exists a countable family of 
cross sections $\{\xi_n(t)\}_{n\in \N}$ of $p$ such that the linear span of $\{\xi_n(t)\}_{n\in \N}$ is dense 
in $E(t)$ for all $t$ and the functions $(0,\infty)\ni t\mapsto \inpr{\xi_m(t)}{\xi_n(t)}$ 
and $(0,\infty)^2\ni (s,t)\mapsto \inpr{\xi_m(s)\xi_n(t)}{\xi_l(s+t)}$ are Borel measurable for all $m,n,l$. 
See \cite[Lemma 7.39]{L2009} for the proof.\\
(2) V. Liebscher \cite[Corollary 7.16]{L2009} showed that for a given $E$ satisfying (i) and (ii), 
there exists at most one Borel structure satisfying (iii). 
This justifies the following definition of isomorphisms of product systems: 
an isomorphism $\theta$ from $E$ to $F$ is a family $\{\theta_t\}_{t>0}$ of unitaries $\theta_t:E(t)\rightarrow F(t)$ satisfying 
$\theta_s(x)\theta_t(y)=\theta_{s+t}(xy)$ for all $x\in E(s)$ and $y\in E(t)$. 
\end{remark}

In what follows, we always consider the Borel structure of $B(H)$ given by the weak operator topology. 
Then $B(H)$ is a standard Borel space with this Borel structure. 
When $\alpha$ is an $E_0$-semigroup, Arveson \cite{Arv89-1} showed that 
\begin{equation*}
\cE_\alpha=\{(t,T)\in (0,\infty)\times B(H);\; T\in \cE_\alpha(t)\}
\end{equation*} 
is a product system with $p(t,T)=t$. 
We will often identify $p^{-1}(t)$ with $\cE_\alpha(t)$ and $(t,T)$ with $T$. 

Cocycle conjugate $E_0$-semigroups give isomorphic product systems. 
Indeed, it is obvious that conjugate $E_0$-semigroups give rise to isomorphic product systems. 
If $U$ is an $\alpha$-cocycle and $\alpha^U$ is the cocycle perturbation of $\alpha$ by $U$, 
then the family of maps $\cE_\alpha(t)\ni T\mapsto U_tT\in \cE_{\alpha^U}(t)$ gives an isomorphism of 
$\cE_\alpha$ and $\cE_{\alpha^U}$. 

Arveson \cite{Arv89-1} showed that the converse is also true. 

\begin{theorem} Two $E_0$-semigroups $\alpha$ and $\beta$ are cocycle conjugate if and only 
the corresponding product systems $\cE_\alpha$ and $\cE_\beta$ are isomorphic. 
\end{theorem}

To sketch the proof of the other implication, we assume that both $\alpha$ and $\beta$ act on $B(H)$ for simplicity, 
and assume that $\cE_\alpha$ and $\cE_\beta$ are isomorphic by $\theta_t:\cE_\alpha(t)\rightarrow \cE_\beta(t)$. 
We choose an arbitrary orthonormal basis $\{V_n\}_{n=1}^\infty$ of $\cE_\alpha(t)$ and set 
\begin{equation*}
U_t=\sum_{n=1}^\infty\theta_t(V_n)V_n^*,
\end{equation*}
which converges to a unitary operator in the strong operator topology. 
Then $U_t$ is independent of the choice of $\{V_n\}_{n=1}^\infty$, which enables us to show that 
$\{U_t\}_{t>0}$ satisfies the 1-cocycle relation. 
Moreover the condition (iii) implies that the map $(0,\infty)\ni t \mapsto U_t$ is Borel. 
Now it follows from a standard trick that $\{U_t\}_{t>0}$ is continuous, and it is an $\alpha$-cocycle. 
By construction, we have $\theta_t(V)=U_tV$ for all $V\in \cE_\alpha(t)$ and $\beta$ is the cocycle perturbation of $\alpha$ 
by $U$ thanks to (\ref{endo}). 

More strongly, Arveson \cite{Arv90-2} showed that the association $\alpha\mapsto \cE_\alpha$ induces a one-to-one correspondence 
between the set of cocycle conjugacy classes of $E_0$-semigroups and the set of isomorphism classes of product systems. 
The only issue now is surjectivity of $\alpha\mapsto \cE_\alpha$. 

\begin{theorem} For any product system $E$, there exists an $E_0$-semigroup $\alpha$ 
whose product system $\cE_\alpha$ is isomorphic to $E$. 
\end{theorem}

Arveson's original proof is really involved and it was the only proof for a while 
(see \cite{Arv2006}, \cite{L2009}, \cite{S2006-1}, \cite{S2006-2} for simpler proofs). 
In order to prove the theorem, Arveson developed the representation theory of product systems, 
which is interesting in its own right. 

A representation $\phi$ of a product system $E$ is a Borel map $\pi:E\rightarrow B(H)$ satisfying 
$\phi(x)\phi(y)=\phi(xy)$ for all $x\in E(s)$, $y\in E(t)$, and $\phi(v)^*\phi(u)1=\inpr{u}{v}$ 
for all $u,v\in E(t)$. 
Arveson constructed the regular representation of $E$ by analogy with the regular representation of a locally compact group.  
If $\phi$ is a representation of $E$, then the following formula with an orthonormal basis $\{e_n\}_{n=1}^\infty$ 
of $E(t)$, 
\begin{equation*}
\alpha_t(A)=\sum_{n=1}^\infty \phi(e_n)A\phi(e_n)^*,\quad A\in B(H),
\end{equation*}
does not depend on the particular choice of $\{e_n\}_{n=1}^\infty$, and 
$\alpha=\{\alpha_t\}_{t>0}$ is a semigroup of endomorphisms of $B(H)$ with appropriate continuity. 
In order to construct an $E_0$-semigroup whose product system is isomorphic to $E$, 
the only problem is that $\alpha_t$ may not be unital. 
The representations with $\alpha_t$ being unital are called essential representations. 
As in the case of locally compact groups, Arveson introduced the spectral $C^*$-algebra $C^*(E)$ of $E$ 
having the universal property with respect to the representations of $E$, and then he constructed a state of $C^*(E)$ 
giving rise to an essential representation of $E$ through the GNS construction. 

Arveson \cite{Arv89-3} showed that $C^*(E)$ is a nuclear $C^*$-algebra for any product system $E$. 
For the structure of $C^*(E)$, see \cite{H2004}, \cite{HZ2003}, \cite{Z2000-1}, \cite{Z2000-2}. 

For attempts to generalize product systems to those for Hilbert $W^*$ and $C^*$-modules, 
see, for example, \cite{Ale2004}, \cite{BBLS2004}, \cite{BSk2001}, \cite{MS2002}, \cite{MS2007}, \cite{S2008}. 

\subsection{CCR flows} The most fundamental examples of $E_0$-semigroups are CCR flows 
acting on $B(H)$, where $H$ is the symmetric Fock space over the test function space $L^2((0,\infty),K)$. 

For a complex Hilbert space $G$, which will be $L^2((0,\infty),K)$ for later use, 
we denote by $e^G$ the symmetric Fock space 
\begin{equation*}
e^G=\bigoplus_{n=0}^\infty G^n,
\end{equation*}
where $G^n$ is the $n$-fold symmetric tensor product of $G$, and $G^0$ is interpreted as the 1-dimensional 
space spanned by a unit vector $\Omega$, called the vacuum. 
The exponential vector $\exp(f)\in e^G$ for $f\in G$ is defined by  
\begin{equation*}\exp(f)=\sum_{n=0}^\infty \frac{1}{\sqrt{n!}}f^{\otimes n},
\end{equation*}
and we have $\inpr{\exp(f)}{\exp(g)}=e^{\inpr{f}{g}}.$
The set of exponential vectors form an independent and total subset of $G$.

When $G$ is decomposed into the direct sum of two closed subspaces $G_1$ and $G_2$, we have 
\begin{align*}
 \inpr{\exp(f_1\oplus f_2)}{\exp(g_1\oplus g_2)}&= e^{\inpr{f_1\oplus f_2}{g_1\oplus g_2}}
 =e^{\inpr{f_1}{g_1}+\inpr{f_2}{g_2}}\\
 &=\inpr{\exp(f_1)}{\exp(g_1)}\inpr{\exp(f_2)}{\exp(g_2)} \\
 &=\inpr{\exp(f_1)\otimes \exp(f_2)}{\exp(g_1)\otimes \exp(g_2)},
\end{align*}
for $f_1,g_1\in G_1$ and $f_2,g_2\in G_2$. 
This shows that the map 
\begin{equation*}
\exp(f_1\oplus f_2)\mapsto \exp(f_1)\otimes \exp(f_2)
\end{equation*} 
extends to a unitary from $e^{G_1\oplus G_2}$ onto $e^{G_1}\otimes e^{G_2}$. 
Therefore forming the symmetric Fock space is a functor from the category of Hilbert spaces into itself 
transforming direct sums into tensor products. 
In what follows we always identify $e^{G_1\oplus G_2}$ with $e^{G_1}\otimes e^{G_2}$. 

We denote by $W(f)\in B(e^G)$ the Weyl operator for $f\in G$, which is the unitary operator defined by 
\begin{equation*}
W(f)\exp(g)=e^{-\frac{1}{2}\|f\|^2-\inpr{g}{f}} \exp(g+f).
\end{equation*}
The Weyl operators satisfy the canonical commutation relation in the Weyl form 
\begin{equation*}
W(f)W(g)=e^{i\Im\inpr{f}{g}}W(f+g),
\end{equation*}
and their linear span is a dense $*$-subalgebra of $B(e^G)$ in the weak operator topology. 

Now we specify the test function space $G$ to be the set of square integrable functions 
$L^2((0,\infty),K)$ on the half-line with values in a complex Hilbert space $K$, called the 
multiplicity space. 
We denote by $S=\{S_t\}_{t\geq 0}$ the (forward) shift semigroup acting on $L^2((0,\infty),K)$: 
\begin{equation*}
S_tf(x)=\left\{
\begin{array}{ll}
0 , &\mathrm{for}\quad 0<x<t  \\
f(x-t) , &\mathrm{for}\quad t\leq x
\end{array}
\right.
.
\end{equation*}

\begin{definition} There exists a unique $E_0$-semigroup $\alpha^K$ acting on $B(H)$ with 
$H=e^{L^2((0,\infty),K)})$ satisfying  
\begin{equation*}
\alpha^K_t(W(f))=W(S_tf),\quad \forall f \in L^2((0,\infty),K).
\end{equation*}
We call $\alpha^K$ the CCR flow of rank $\dim K$. 
\end{definition}
 
\begin{remark} The CAR flows are defined in the same way except for replacing the symmetric Fock space 
with the antisymmetric Fock space. 
Powers-Robinson \cite{PR89} showed that the CCR flow and CAR flow of the same rank are conjugate. 
\end{remark}

Arveson \cite{Arv89-1} identified the product systems corresponding to the CCR flows.  
For $0\leq a<b\leq \infty$, we regard $L^2((a,b),K)$ as a closed subspace of $L^2((0,\infty),K)$ in a natural way, 
and $e^{L^2((a,b),K)}$ as a closed subspace of $H=e^{L^2((0,\infty),K)}$ generated by 
$\{\exp(f)\}_{f\in L^2((a,b),K)}$. 
We denote by $\Gamma(S_t)$ the isometry in $B(H)$ determined by $\Gamma(S_t)\exp(f)=\exp(S_tf)$ for all 
$f\in L^2((0,\infty),K)$. 
For $t>0$, we set $E^K(t)=e^{L^2((0,t),K)}\subset H$. 
Then 
\begin{equation*} 
E^K=\{(t,\xi)\in (0,\infty)\times H;\; \xi\in E^K(t)\}
\end{equation*}
is a product system with $p(t,\xi)=t$ and multiplication 
\begin{equation*}(s,\xi)\cdot(t,\eta)=(s+t,\xi\otimes \Gamma(S_s)\eta),
\end{equation*}
where we use the following identification 
\begin{equation*}
e^{L^2((0,s+t),K)}=e^{L^2((0,s),K)\oplus S_sL^2((0,t),K)}=e^{L^2((0,s),K)}\otimes \Gamma(S_s)e^{L^2((0,t),K)}.
\end{equation*}
We call $E^K$ the exponential product system, which is isomorphic to the product system $\cE_{\alpha^K}$ associated with 
the CCR flow $\alpha^K$, via the representation $\phi:E^K\rightarrow B(H)$ given by 
$\phi((t,\xi))\eta=\xi\otimes \Gamma(S_t)\eta$ for $\eta\in H$.  

\begin{remark}\label{white noise} The exponential product systems can be interpreted as product systems associated 
with white noise.  
For simplicity, we consider the case $K=\C$. 
Let $\{B_t\}_{t\geq 0}$ be the standard Brownian motion defined on the probability space $(\Omega,\cF,W)$. 
For $0\leq s<t\leq \infty$, we denote by $\cF_{s,t}$ the $\sigma$-algebra generated by the increments 
$B_v-B_u$ for all $s\leq u<v\leq t$. 
We may assume $\cF=\cF_{0,\infty}$. 
Then the well-known Wiener-Ito chaos decomposition (see \cite[Chapter IV]{M93}) says that the nested system of subspaces 
$\{L^2(\cF_{s,t})\}_{0\leq s<t\leq \infty}$ 
of $L^2(\Omega,\cF,W)$ is identified with that of subspaces $\{e^{L^2(s,t)}\}_{0\leq s<t\leq \infty}$ of $e^{L^2(0,\infty)}$, 
and the identification goes along with 
the tensor product factorizations 
$L^2(\cF_{r,t})=L^2(\cF_{r,s})\otimes L^2(\cF_{s,t})$ and $e^{L^2(r,t)}=e^{L^2(r,s)}\otimes e^{L^2(s,t)}$ for 
$r<s<t$. 
Moreover, the time shift of the Brownian motion induces an isometry from $L^2(\cF_{0,\infty})$ onto $L^2(\cF_{t,\infty})$. 
Therefore we can completely describe the exponential product system in terms of so called white noise, which consists of 
$(\Omega,\cF,W)$, $\{\cF_{s,t}\}_{0\leq s<t\leq \infty}$, and the time shift (strictly speaking, 
white noise is the two-sided version of it). 
White noise is only a special example of Tsirelson's notion of noises, and this interpretation opened up 
Tsirelson's probabilistic approach to $E_0$-semigroups (see Section 3). 
\end{remark}

\subsection{Index}
It is natural to ask whether one can distinguish the CCR flows with different ranks up to cocycle conjugacy. 
To answer the question in the positive, we need an isomorphism invariant for product systems, and the first 
such invariant was provided by Arveson \cite{Arv89-1}. 

\begin{definition} Let $E$ be a product system. 
A unit of $E$ is a non-zero measurable section $(0,\infty)\ni t\mapsto u(t)\in E(t)$ that is multiplicative, 
\begin{equation*}
u(s+t)=u(s)u(t),\quad \forall s,t>0.
\end{equation*}
We denote by $\cU_E$ the set of units of $E$.  
\end{definition} 

For the product system $\cE_\alpha$ associated with an $E_0$-semigroup $\alpha$, a unit is nothing but 
a continuous semigroup $V=\{V_t\}_{t\geq 0}$ of isometries, up to normalization, satisfying the intertwining property 
$V_tA=\alpha_t(A)V_t$. 
We say that $\alpha$ is spatial if $\cU_{\cE_\alpha}$ is not empty.

The multiplicative property of units implies that for $u,v\in \cU_E$, there exists a unique complex number 
$c_E(u,v)$ satisfying $\inpr{u(t)}{v(t)}=e^{tc_E(u,v)}$ for all $t>0$. 
For each fixed $t>0$, the function $e^{tc_E(u,v)}$ on $\cU_E\times \cU_E$ is positive definite by definition, and 
the Schoenberg theorem shows that $c_E$ is conditionally positive definite. 
Let $\C_0\cU_E$ be the set of functions $\xi:\cU_E\rightarrow \C$ with finite support and $\sum_{u\in \cU_E}\xi(u)=0$. 
Then $c_E$ being conditionally positive definite means 
\begin{equation*}
\sum_{u,v}c_E(u,v)\xi(u)\overline{\xi(v)}\geq 0,\quad \forall \xi\in \C_0\cU_E.
\end{equation*}
Therefore $c_E$ gives a positive semi-definite inner product of $\C_0\cU_E$, and we get a separable Hilbert space, 
denoted by $H(\cU_E,c_E)$, by the usual procedure. 
The dimension of this Hilbert space is an isomorphism invariant of the product system $E$. 
Note that $H(\cU_E,c_E)=\{0\}$ is possible, not as in the case of positive definite functions.  
Intuitively $\dim H(\cU_E,c_E)$ is $``\dim\cU_E"-1$.  

\begin{definition} The index of a product system $E$ with $\cU_E\neq \emptyset$ is 
\begin{equation*}\ind(E)=\dim H(\cU_E,c_E).\end{equation*} 
The index of a spatial $E_0$-semigroup $\alpha$ is $\ind(\alpha)=\ind(\cE_\alpha)$. 
\end{definition}

\begin{remark} The first attempt to introduce a numerical invariant for $E_0$-semigroups was made by Powers \cite{P88}. 
To define his index, he constructed what is now called the boundary representation by an infinitesimal argument. 
However, his definition a priori depends on the choice of a normalized unit. 
Powers-Price \cite{PP90} clarified the precise relationship between Arveson's index and the boundary representation,   
and Alevras \cite{Ale95} showed that the boundary representation does not depend on the choice of a normalized unit. 
\end{remark}

Arveson \cite{Arv89-2} showed that the addition formula $\ind(\alpha\otimes \beta)=\ind(\alpha)+\ind(\beta)$ holds 
for spatial $E_0$-semigroups $\alpha$ and $\beta$.  
When one of them is non-spatial, so is $\alpha\otimes \beta$. 

For a product system with $\cU_E\neq \emptyset$, we fix a unit $e\in \cU_E$ with normalization 
$\inpr{e(t)}{e(t)}=1$, and set 
\begin{equation*}\cU_E^e=\{u\in \cU_E;\; \inpr{u(t)}{e(t)}=1,\;\forall t>0\}.\end{equation*}
Let $L_u=\delta_u-\delta_e\in \C_0\cU_E$ for $u\in \cU_E^e$. 
Then one can show that $\inpr{L_u}{L_v}=c_E(u,v)$ and $H(\cU_E,c_E)$ is spanned by $\{L_u\}_{u\in \cU_E^e}$. 

For the exponential product systems $E^K$, we set 
\begin{equation*}
u^{(a,\zeta)}(t)=e^{at}\exp(1_{(0,t)}\otimes \zeta)\in E^K(t),
\end{equation*}
for $a\in \C$ and $\zeta\in K$, where $1_{(0,t)}$ is the indicator function of the interval $(0,t)$. 
Then $u^{(a,\zeta)}$ is a unit and 
\begin{equation*}c_{E^K}(u^{(a_1,\zeta_1)},u^{(a_2,\zeta_2)})
=a_1+\overline{a_2}+\inpr{\zeta_1}{\zeta_2}.\end{equation*} 
Arveson \cite{Arv89-1} showed that there are no other units. 
We can choose $e=u^{(0,0)}$, the vacuum vector, and in this case 
$\cU_{E^K}^e=\{u^{(0,\zeta)};\; \zeta\in K\}$.  
Now the Hilbert space $H(\cU_{E^K},c_{E^K})$ is identified with the multiplicity space $K$. 

\begin{theorem} Every unit of the exponential product system $E^K$ is of the form $u^{(a,\zeta)}$ for 
$a\in \C$ and $\zeta\in K$. The correspondence 
$K\ni \zeta\mapsto \delta_{u^{(0,\zeta)}}-\delta_{u^{(0,0)}}\in H(\cU_{E^K},c_{E^K})$ 
gives a unitary operator from $K$ onto $H(\cU_{E^K},c_{E^K})$. 
In particular, the index $\ind(\alpha^K)$ of the CCR flow $\alpha^K$ is $\dim K$. 
\end{theorem}

\subsection{Type classification and the classification of type I product systems} 

Product systems are classified according to how abundant the set $\cU_E$ is. 

\begin{definition} Let $E$ be a product system. 
\begin{itemize}
\item [(i)] We say that $E$ is of type I if the linear span of 
\begin{equation*}\{u_1(t_1)u_2(t_2)\cdots u_n(t_m)\in E(t);\; u_1,u_2,\ldots, u_m\in \cU_E,\;t_1+t_2+\cdots t_m=t\}\end{equation*}
is dense in $E(t)$ for all (or equivalently, some) $t>0$. 
Type I product systems are further divided into type I$_n$, $n=1,2,\ldots,\infty$, according to 
the value $n=\ind(E)$ of the index. 
\item [(ii)] We say that $E$ is of type II if $\cU_E\neq \emptyset$ and the condition in (i) is not satisfied. 
Type II product systems are further divided into type II$_n$, $n=0,1,\ldots,\infty$, according to 
the value $n=\ind(E)$ of the index. 
\item [(iii)] We say that $E$ is of type III if $\cU_E=\emptyset$. 
\end{itemize}
We use the same terms for $E_0$-semigroups $\alpha$ through the product systems $\cE_\alpha$. 
\end{definition}

\begin{remark} Since we assume that $E(t)$ is infinite dimensional for each $t>0$, type I$_0$ 
never occurs while type II$_0$ product systems actually occur and they form an important subclass of 
type II product systems. 
One could define trivial $E_0$-semigroups (i.e. $\alpha_t\in \Aut(B(H))$) to be of type I$_0$.  
\end{remark}

Arveson \cite{Arv89-1} completely classified type I $E_0$-semigroups. 

\begin{theorem}\label{typeI} Let $E$ be a type I product system, and let $K=H(\cU_E,c_E)$. 
Then $E$ is isomorphic to the exponential product system $E^K$. 
In particular, there exists exactly one cocycle conjugacy class of $E_0$-semigroups of type I$_n$ 
for each $n=1,2,\ldots, \infty$. 
\end{theorem}

To prove that two given $E_0$-semigroups are cocycle conjugate, in general it is not so easy to construct a cocycle 
explicitly, and Arveson's proof really makes use of the advantage of introducing the abstract notion of product systems. 

We sketch how to construct the isomorphism in Theorem \ref{typeI}. 
For a given type I product system $E$, we fix a normalized unit $e\in \cU_E$, and define $\cU_E^e$ and 
$L_u$ for $u\in \cU_E^e$ as before. 
Let $u_1,u_2,\ldots, u_m\in \cU_E^e$, and let $t_1,t_2,\ldots,t_m>0$ with summation $t$. 
We set $s_0=0$, and $s_i=t_1+t_2+\cdots+t_i$. 
Then the isomorphism in Theorem \ref{typeI} takes $u_1(t_1)u_2(t_2)\cdots u_m(t_m)\in E(t)$ to 
\begin{equation*}\exp(\sum_{i=1}^m1_{(s_{i-1},s_i)}\otimes L_{u_i})\in e^{L^2((0,t),K)}.\end{equation*}

Theorem \ref{typeI} shows that even if we try more general L{\'e}vy processes instead of the Brownian motion 
in Remark \ref{white noise}, we still get exponential product systems.  
To obtain non-type I product systems, we need truly non-classical noises. 

Later, Arveson \cite{Arv97-3} strengthened Theorem \ref{typeI}. 
Let $E$ be a product system. 
We say that a non-zero vector $x\in E(t)$ is decomposable if for every $0<s<t$, there exist $y\in E(s)$ and 
$z\in E(t-s)$ satisfying $x=yz$.  
We denote by $D(t)$ the set of decomposable vectors in $E(t)$. 
We say that $E$ is decomposable if $D(t)$ is a total subset of $E(t)$ for all $t>0$. 
A typical example of a decomposable vector is the product of units $u_1(t_1)u_2(t_2)\cdots u_m(t_m)$ as above, 
and so type I product systems are decomposable. 
Arveson \cite{Arv97-3} showed that every decomposable product system is of type I, 
whose proof is much more involved than that of Theorem \ref{typeI}.

\subsection{Gauge groups} 
It is often true that significant information of a mathematical object is carried by the 
structure of its automorphism group. 
For an $E_0$-semigroup $\alpha$, the automorphism group $\Aut(\cE_\alpha)$ of the associated product system 
$\cE_\alpha$ is isomorphic to the gauge group. 

\begin{definition} A gauge cocycle of an $E_0$-semigroup $\alpha$ is an $\alpha$-cocycle 
$U$ satisfying $\alpha^U=\alpha$, that is $U_t\in \alpha_t(B(H))'$ for all $t>0$. 
We denote by $G(\alpha)$ the group of gauge cocycles, and call it the gauge group of $\alpha$. 
\end{definition}
Since $U(H)$ is a Polish group in the weak operator topology, so is the gauge group $G(\alpha)$  
in the topology of uniform convergence on compact subsets. 

For a type I product system, automorphisms are determined by their actions on units. 
Arveson \cite{Arv89-1} completely determined the structure of the gauge groups of the CCR flows. 

\begin{theorem}\label{gauge} For the CCR flow $\alpha^K$, the gauge group $G(\alpha^K)$ is isomorphic to 
the central extension of the semi-direct product group $K\rtimes U(K)$ by $\R$. 
More precisely $G(\alpha^K)=\R\times K\times U(K)$ as a topological space, and the group operation is given by 
\begin{equation*}(\lambda,\xi,U)(\mu,\eta,V)=(\lambda+\mu+\omega(\xi,U\eta),\xi+U\eta,UV),
\end{equation*}
where $\omega$ is the symplectic form $\omega(\xi,\eta)=\Im\inpr{\xi}{\eta}$, $\xi,\eta\in K$. 
\end{theorem}

Since every type I $E_0$-semigroup is cocycle conjugate to one of the CCR flows, Theorem \ref{gauge} is often very useful 
in order to show type I criteria in specific constructions of $E_0$-semigroups 
(see, for example, \cite{I2007},\cite{IS2008},\cite{IS2010}). 

\subsection{Dilation theory}
One of the richest sources of $E_0$-semigroups is semigroups of unital normal completely positive maps, which 
are often easier to construct than $E_0$-semigroups. 

\begin{definition} 
A $CP_0$-semigroup is a family of unital normal completely positive maps  $P=\{P_t\}_{t\geq 0}$ of 
a von Neumann algebra $N$ satisfying the following conditions: 
\begin{itemize}
\item[(i)] $P_0=\id$, 
\item[(ii)] $P_s\circ P_t=P_{s+t}$ for all $s,t\geq 0$, 
\item[(iii)] For every $a\in N$ and every normal functional $\varphi\in N_*$, 
the function $[0,\infty)\ni t\mapsto \varphi(P_t(a))\in \C$ is continuous. 
\end{itemize}
\end{definition}

A corner $N$ of a von Neumann algebra $M$ is a von Neumann subalgebra of the particular form $N=pMp$, 
where $p$ is a projection. 
The central carrier of $p$ in $M$ is the smallest projection in the center $Z(M)$ of $M$ dominating $p$. 

\begin{definition} A dilation of a $CP_0$-semigroup $P$ acting on $N$ consists of a von Neumann algebra $M$, 
a projection $p\in M$, and an $E_0$-semigroup $\alpha$ acting on $M$ satisfying 
\begin{itemize}
\item[(i)] $N$ is the corner $pMp$ of $M$, 
\item[(ii)] $P_t(a)=p\alpha_t(a)p$ for any $a\in N$ and $t\geq 0$. 
\end{itemize}
If moreover the following two conditions are satisfied, we say that the dilation is minimal:   
\begin{itemize}
\item[(iii)] $M$ is generated by $\cup_{t\geq 0}\alpha_t(N)$, 
\item[(iv)] the central carrier $c(p)$ of $p$ in $M$ is $1_M$. 
\end{itemize} 
\end{definition}
Note that (ii) implies that $\{\alpha_t(p)\}_{t\geq 0}$ is an increasing family of projections, 
and if (iii) is satisfied, it converges to the unit of $M$ in the strong operator topology as $t$ tends to 
$\infty$. 
In the case of a minimal dilation, if $N$ is a factor, so is $M$. 

We identify two dilations $(M,p,\alpha)$ and $(R,q,\beta)$ if there exists an isomorphism $\theta$ from $M$ 
onto $R$ such that the restriction of $\theta$ to $pMp=qRq=N$ is the identity map and $\theta\circ \alpha_t=\beta_t\circ \theta$ holds 
for any $t\geq 0$. 

\begin{theorem}\label{dilation} There exists a unique minimal dilation for any $CP_0$-semigroup 
acting on a von Neumann algebra with separable predual. 
\end{theorem}

Bhat \cite{B96} proved Theorem \ref{dilation} in the case of type I factors based on his previous results 
\cite{BP94},\cite{BP95}, and he computed the product systems of the dilations. 
The existence in the general case was obtained in \cite{B99},\cite{BSk2001} 
(see also \cite{MS2002},\cite{MS2007}). 
These works use the fact that the map $[0,\infty)\ni t\mapsto P_t(a)\in N$ is continuous in the strong operator topology, 
which was proved by Markiewicz and Shalit \cite{MS2010} later.
Arveson \cite{Arv97-1},\cite[Section 8.9]{Arv2003-1} showed the uniqueness in the general case.  
The uniqueness in the non-unital case appears very subtle (see \cite{B2003}). 

When $P$ is a $CP_0$-semigroup acting on a type I factor, Arveson \cite{Arv97-2} described 
the units of the minimal dilation in terms of $P$. 
Answering a question raised in \cite{B96} about the dilations of $CP_0$-semigroups acting on matrix algebras, 
Arveson \cite{Arv99} showed the following result (cf. \cite{P99-2}). 

\begin{theorem} Let $P$ be a $CP_0$-semigroup acting on $B(H_0)$, with $H_0$ possibly finite dimensional, 
which is not a semigroup of automorphisms. 
If the generator $L$ of $P$ is bounded, then the minimal dilation $\alpha$ of $P$ is an $E_0$-semigroup of type I,  
and the index $\ind(\alpha)$ is the rank of $L$. 
\end{theorem}
For the definition of the rank of $L$, see \cite[Chapter 10]{Arv2003-1}. 

Arveson \cite{Arv97-4},\cite{Arv2000} applied dilation theory to what is called interaction theory, 
which, roughly speaking, deals with coupling of two $E_0$-semigroups, one for the past and the other for the future, 
with prescribed invariant normal states. 

Markiewicz \cite{Mar2003} computed the product systems for the minimal dilations of concrete $CP_0$-semigroups 
acting on $B(L^2(\R))$ arising from a modified Weyl-Moyal quantization of convolution semigroups of probability 
measures on $\R^2$, including the CCR heat flow discussed by Arveson \cite{Arv2002-1} as a special case. 
Despite that the generators of these $CP_0$-semigroups are unbounded, the resulting product systems 
are still of type I (in fact type I$_2$). 

Shalit-Solel \cite{SS2009} and Bhat-Mukherjee \cite{BM2010} recently introduced essentially the same notion, 
called subproduct systems in \cite{SS2009}, and inclusion systems in \cite{BM2010}, 
which had been implicitly used in \cite{BSk2001}, \cite{Mar2003}, \cite{MS2002}. 
Their role in product systems is somewhat similar to the role of $CP_0$-semigroups in $E_0$-semigroups. 
See \cite{BLMS2011}, \cite{M2011}, \cite{P2004}, \cite{S2010} for related results.

\section{Beyond Arveson: Type II case}
The first example of an $E_0$-semigroup of type II was constructed implicitly by Tsirelson-Vershik \cite{TV98} 
in 1998 via the noise theory, and about the same time by Powers \cite{P99-1} via the boundary representation. 
Later on, both Powers \cite{P2003-1} and Tsirelson \cite{T2003-1} constructed uncountably many type II$_0$ examples. 
Thanks to Arveson's addition formula, we have $\mathrm{I}_n\otimes\mathrm{II}_0=\mathrm{II}_n$, and so 
type II$_n$ examples exist for $n=0,1,\ldots,\infty$. 
 
\subsection{Tsirelson's probabilistic method}
Tsirelson introduced the following concept of a noise in probability theory (see \cite{T2003-1},\cite{T2004}).

\begin{definition} A noise consists of a probability space $(\Omega, \cF,P)$, sub-$\sigma$-fields $\cF_{s,t}\subset \cF$ 
for $s,t\in \R$ with $s<t$, and a measure preserving action $T$ of $\R$ on the probability space satisfying 
\begin{itemize}
\item[(i)] $\cF_{r,s}\otimes \cF_{s,t}=\cF_{r,t}$ for any $r<s<t$, 
\item[(ii)] $T_h$ sends $\cF_{s,t}$ to $\cF_{s+h,t+h}$, 
\item[(iii)] $\cF$ is generated by $\cup_{s<t}\cF_{s,t}$, 
\item[(iv)] $P(A \ominus T_h^{-1}(A))\to 0$ as $h\to0$ for any $A\in \cF$, where $A\ominus B$ is 
the symmetric difference of $A$ and $B$.  
\end{itemize}
\end{definition}

A typical example of a noise is white noise already discussed in Remark \ref{white noise}. 
Every noise gives rise to a product system by $E(t)=L^2(\cF_{0,t})$ and 
$\xi\cdot \eta=\xi(\eta\circ T_{-s})$ for $\xi\in E(s)$, $\eta\in E(t)$. 
The resulting product system has at least one unit given by the constant function $1$. 
Using the factorization $L^2(\cF_{-\infty,\infty})=L^2(\cF_{-\infty,0})\otimes L^2(\cF_{0,\infty})$ and the 
1-parameter unitary group $\{U_t\}_{t\in \R}$ arising from $T$, we can directly construct the corresponding 
$E_0$-semigroup $\alpha$ acting on the type I factor $B(\cF_{0,\infty})$ by 
$1\otimes \alpha_t(A)=U_t(1\otimes A)U_t^*$ for $A\in B(L^2(\cF_{0,\infty}))$. 

A noise arising from a L{\'e}vy process is called a classical noise, and Tsirelson showed that 
every noise contains the maximal classical noise, called the classical part, 
which corresponds to the subspaces generated by decomposable vectors in $E(t)$. 
A black noise is a noise with trivial classical part, which gives rise to a product system of type II$_0$. 
Tsirelson-Vershik \cite{TV98} showed 

\begin{theorem} There exists a black noise. 
\end{theorem}

A black noise is a singular object, and is not so easy to construct. 
Very few examples are known. 

Like the Wiener-Ito chaos decomposition, the space $L^2(\Omega,\cF,P)$ has a canonical decomposition. 
However, subspaces for a finite number of particles do not generate the whole space unless the noise is classical. 
From this decomposition, random sets arise as an invariant of the noise, and it also makes sense as an 
invariant of product systems of type II$_0$. 
A variant of this invariant adapted to type II product systems was extensively studied by Liebscher \cite{L2009}. 
Among others he showed that there exists a type II$_n$ product system for each $n=0,1,\ldots,\infty$, that never 
splits as a tensor product of two product systems. 

Tsirelson \cite{T2003-1} introduced the notion of homogeneous continuous products of measure classes (HCPMC), 
more general objects than noises. 
This idea originated from Vershik according to Tsirelson. 
HCPMCs are more flexible than noises, and are still good enough to produce $E_0$-semigroups. 
The main difference of an HCPMC from a noise is that the independence $\cF_{r,s}\otimes \cF_{s,t}=\cF_{r,s}$ 
does not necessarily hold for $P$, but it is required to hold for a measure equivalent to $P$. 
Typical examples of HCPMCs arise from random sets associated with Markov processes. 
Tsirelson \cite{T2003-1} showed that they give uncountably many type II$_0$ product systems that are not anti-isomorphic to themselves.

Tsirelson \cite{T2008} recently constructed a type II$_1$ product system, by using the noise theory, whose automorphism group 
does not act on the set of normalized units transitively. 
This shows that choosing an arbitrary unit is not always justified in order to define an invariant of product systems. 

\subsection{Powers $CP$-flows}
Powers \cite{P2003-1},\cite{P2003-2} found a systematic way to construct $E_0$-semigroups of type II by using 
dilation theory. 
For simplicity, we assume that $CP$-flows are unital in this note, though non-unital ones also play important roles in 
Powers's argument. 

\begin{definition} Let $K$ be a separable complex Hilbert space, and let $H_0=L^2((0,\infty),K)$. 
We denote by $S=\{S_t\}_{t\geq 0}$ the shift semigroup acting on $H_0$. 
A $CP$-flow $P$ is a $CP_0$-semigroup acting on $B(H_0)$ satisfying $S_tA=P_t(A)S_t$ for all $A\in B(H_0)$ and $t\geq 0$. 
\end{definition}

The minimal dilation of a $CP$-flow always has a unit. 
On the other hand, any spatial $E_0$-semigroup is cocycle conjugate to an $E_0$-semigroup that is a $CP$-flow. 
Thus it is important to understand the structure of $CP$-flows. 
Powers \cite{P2003-1} showed that all of the information of a $CP$-flow is encoded in its boundary weight map. 

We define $\Lambda:B(K)\rightarrow B(H_0)$ by 
\begin{equation*}\Lambda(A)f(x)=e^{-x}Af(x),\quad f\in H_0.\end{equation*}
For simplicity, we include complete positivity and the unitality condition in the definition of boundary weight maps. 

\begin{definition} Let $\mathfrak{A}(H_0)=(1_{H_0}-\Lambda(1_K))^{1/2}B(H_0)(1_{H_0}-\Lambda(1_K))^{1/2}$. 
A boundary weight $\mu$ is a linear functional of $\mathfrak{A}(H_0)$ such that the linear functional 
\begin{equation*}B(H_0)\ni A\mapsto \mu((1_{H_0}-\Lambda(1_K))^{1/2}A(1_{H_0}-\Lambda(1_K))^{1/2})\in \C,\end{equation*} 
is bounded and normal. 
We denote by $\mathfrak{A}(H_0)_*$ the set of boundary weights. 
A boundary weight map $\omega$ is a completely positive map $\omega:B(K)_*\rightarrow \mathfrak{A}(H_0)_*$ satisfying 
$\omega(\rho)(1_{H_0}-\Lambda(1_K))=\rho(1_K)$ for any $\rho\in B(K)_*$. 
\end{definition}

For a normal map $\Phi$ between von Neumann algebras, we denote by $\hat{\Phi}$ the map between the preduals 
induced by $\Phi$. 
Being a semigroup, a $CP$-flow $P$ is determined by its resolvent 
\begin{equation*}R_P(A)=\int_0^\infty e^{-t}P_t(A)dt.\end{equation*}
On the other hand, since we have 
\begin{equation*}P_t(A)=S_tAS_t^*+(1-S_tS_t^*)P_t(A)(1-S_tS_t^*),\quad A\in B(H_0),\end{equation*} 
the first approximation of $R_P$ is 
\begin{equation*}\Gamma(A)=\int_0^\infty e^{-t}S_tAS_t^*dt.\end{equation*}
Our task is to describe the difference $R_P-\Gamma$, which is a completely positive map. 
Powers \cite{P2003-1} showed the following. 

\begin{theorem} Let the notation be as above. 
\begin{itemize}
\item[(i)] For any $CP$-flow $P$, there exists a unique boundary weight map $\omega$ satisfying  
\begin{equation*}\hat{R_P}(\eta)=\hat{\Gamma}(\omega(\hat{\Lambda}\eta)+\eta),\quad \forall \eta\in B(H_0)_*.\end{equation*} 
The map $\omega$ is called the boundary weight map associated with $P$. 
\item[(ii)] For a boundary weight map $\omega$, we set $\omega_t(\rho)(A)=\omega(\rho)(S_tS_t^*AS_tS_t^*)$. 
If $\id+\hat{\Lambda}\omega_t$ is invertible and $\hat{\pi}_t:=\omega_t\circ(\id+\hat{\Lambda}\omega_t)^{-1}$ 
is a completely positive contraction for any $t>0$, then $\omega$ is the boundary weight map associated with a $CP$-flow.  
\end{itemize}
\end{theorem}

When $K=\C$, a boundary weight map $\omega$ is identified with the boundary weight $\omega(1_K)$, which is  
a normal semifinite weight of $B(H_0)$ satisfying 
\begin{equation*}
\omega(1_K)(1_{H_0}-\Lambda(1_K))=1.
\end{equation*}
The condition in (ii) is automatically satisfied for such a weight. 
In particular, any function $f$ with 
\begin{equation*}\int_0^\infty |f(x)|^2(1-e^{-x})dx=1,\end{equation*} gives rise to a $CP$-flow. 
Already this case provides uncountably many type II$_0$ $E_0$-semigroups. 
More precisely, Powers showed that the resulting $E_0$-semigroup is of type II$_0$ unless $f\in L^2(0,\infty)$, 
and that for any such functions $f_1,f_2$ as above, the resulting $E_0$-semigroups are 
cocycle conjugate if and only if $c_1f_1+c_2f_2 \in L^2(0,\infty)$ for some $c_1,c_2\in \C\setminus \{0\}$.

See \cite{APP2006},\cite{J2010},\cite{JM2012},\cite{JMP2012},\cite{MP2009} for recent progress in this approach. 
There are several $E_0$-semigroups of type II whose gauge groups are known. 

It is desirable to unify the two approaches presented in this section, and we propose two problems, just to start with. 
\begin{itemize}
\item[(i)] Compute Tsirelson's random sets invariant for an $E_0$-semigroup of type II$_0$ arising from a $CP$-flow  
in terms of its boundary weight map. 
\item[(ii)] Give a description of the boundary representation for an $E_0$-semigroup of type II arising from an HCPMC.
\end{itemize} 
\section{Beyond Arveson: Type III case}
Powers \cite{P87} constructed the first example of a type III $E_0$-semigroup using the CAR algebra in 1987, 
just after the theory of $E_0$-semigroups was initiated. 
It had been the only example of a non-type I $E_0$-semigroup for a while. 
Much later, Tsirelson \cite{T2003-1} constructed the first continuous family of type III product systems 
using HCPMCs coming from off-white noises. 

\subsection{CAR construction}
Recall that the CAR flows are conjugate to the CCR flows, which are necessarily of type I, and they are constructed in 
the vacuum representation. 
To construct the first type III example, Powers used a quasi-free representation of the CAR algebra instead of the vacuum representation. 

Let $G:=L^2((0,\infty),\C^N)$. 
We denote by $\fA$ the CAR algebra over the test function space, 
which is the universal $C^*$-algebra generated by 
$\{a(f); f \in G\}$, depending linearly on $f$, and satisfying the CAR relations:  
\begin{eqnarray*}
a(f)a(g) +a(g)a(f)&  = & 0, \\
a(f)a(g)^* +a(g)^*a(f) & = & \langle f,g\rangle1. 
\end{eqnarray*} 

Let $S=\{S_t\}_{t\geq 0}$ be the shift semigroup acting on $G$. 
Since the CAR relation involves only the inner product, there exists a continuous semigroup $\gamma$ 
of unital endomorphisms of $\fA$ given by $\gamma_t(a(f))=a(S_tf)$. 
If $\pi$ is a type I factor representation of $\fA$ such that $\pi\circ \gamma_t$ is quasi-equivalent to $\pi$ for all $t\geq 0$, 
then $\gamma$ extends to an $E_0$-semigroup acting on $\pi(\fA)''$. 
The vacuum representation is an example of such a representation, giving the CAR/CCR flow of rank $N$.

A quasi-free state $\omega_A$ on $\fA$ associated with a positive contraction $A \in B(G)$  
is a unique state determined by the formula 
\begin{equation*}\omega_A(a(f_n) \cdots a(f_1)a(g_1)^* \cdots a(g_m)^* ) = \delta_{n,m} \det (\langle Af_i, g_j\rangle) .\end{equation*}   
If a positive contraction $A$ satisfies the condition 
\begin{equation}\label{extension}
\tr(A-A^2)<\infty, \quad S_t^*AS_t=A,\quad \forall t\geq 0,
\end{equation}
the GNS representation for $\omega_A$ has the desired property, and we can construct an $E_0$-semigroup. 

To present the positive contraction Powers constructed, we need to introduce Toeplitz operators.  
We regard $G$ as a closed subspace of $\tG:=L^2(\R,\C^N)$, and we denote by $P_+$ the projection from $\tG$ onto $G$. 
For $\Phi\in L^\infty(\R)\otimes M_N(\C)$, we define the corresponding Fourier multiplier $C_\Phi\in B(\tG)$ by 
\begin{equation*}
\widehat{(C_\Phi f)}(p)=\Phi(p)\widehat{f}(p).
\end{equation*}
Then the Toeplitz operator $T_\Phi\in B(G)$ with a symbol $\Phi$ is defined by $T_\Phi f=P_+C_\Phi f$, $f\in G.$

Powers \cite{P87} came up with a mysterious symbol giving a type III example. 

\begin{theorem} Let $N=2$, and let 
$$\Phi(p)=\frac{1}{2}\left(
\begin{array}{cc}
1 &e^{i\theta(p)}  \\
e^{-i\theta(p)} &1 
\end{array}
\right),\quad \theta(p)=(1+p^2)^{-1/5}.
$$
Then $A=T_\Phi$ satisfies the condition (\ref{extension}), and the quasi-free representation for $A=T_\Phi$ gives a type III $E_0$-semigroup.  
\end{theorem}

Arveson \cite[Section 13.3]{Arv2003-1} determined the most general form of a positive contraction $A\in B(G)$ satisfying 
the condition (\ref{extension}), and showed that such an operator must be a Toeplitz operator $T_\Phi$ with a symbol $\Phi$ 
satisfying a certain condition. 
We call the resulting $E_0$-semigroup the Toeplitz CAR flow arising from the symbol $\Phi$. 

Let $\Phi_\nu$ be the matrix valued function given by the same formula as the Powers symbol except for $\theta(p)=(1+p^2)^{-\nu}$. 
Then $T_{\Phi_\nu}$ satisfies the condition (\ref{extension}) for all $\nu>0$, and we denote by $\alpha^\nu$ the resulting $E_0$-semigroup. 
Recently Srinivasan and the author \cite{IS2010} showed the following. 

\begin{theorem} Let the notation be as above. 
\begin{itemize}
\item [(i)] If $\nu>1/4$, then $\alpha^\nu$ is of type I$_2$. 
\item [(ii)] If $0<\nu\leq 1/4$, then $\alpha^\nu$ is of type III. 
\item [(iii)] If $0<\nu_1<\nu_2\leq 1/4$, then $\alpha^{\nu_1}$ and $\alpha^{\nu_2}$ are not cocycle conjugate. 
\end{itemize}
\end{theorem}
To distinguish $\alpha^\nu$ in the type III region $0<\nu\leq 1/4$, we used the type I factorizations of Araki-Woods \cite{AW} 
arising from local von Neumann algebras for the product systems (see the next subsection). 
In general Toeplitz CAR flows are either of type I or of type III. 


\subsection{Off-white noises and generalized CCR flows}
Let $\{B_t\}_{t\in \R}$ be the (two-sided) Brownian motion, and let $X(t)$ be the formal derivative $\frac{dB(t)}{dt}$. 
Then $\{X(t)\}_{t\in \R}$ is a stationary Gaussian generalized (i.e. distribution valued) random process with correlation function 
$E(X(s)X(t))=\delta(s-t)$. 
There is no relation between the past and the future at all. 
Let $\cF_{s,t}$ be the $\sigma$-field generated by $\inpr{X}{f}$ with test functions $f$ supported in $(s,t)$. 
Then we get white noise. 

An off-white noise is an HCPMC, not really a noise, constructed in the same way by replacing $X(t)$ with a stationary 
Gaussian generalized random process $\xi(t)$ having a slight correlation of the past and the future. 
The correlation function $C(s-t)=E(\xi(s)\xi(t))$ is now a positive definite distribution, whose Fourier transform 
$\hat{C}$ is a measure. 
Tsirelson \cite{T2002},\cite{T2003-2} showed that if $\hat{C}$ has a density $\sigma(\lambda)$ with respect to 
the Lebesgue measure $d\lambda$, and $\sigma$ satisfies 
\begin{equation*}
\int_{\R^2}\frac{|\log \sigma(\lambda_1)-\log \sigma(\lambda_2)|^2}{|\lambda_1-\lambda_2|^2}d\lambda_1d\lambda_2<\infty,
\end{equation*}  
then we get an HCPMC, which is called an off-white noise. 
The function $\sigma(\lambda)$ is called the spectral density function of the off-white noise, and all information 
about the off-white noise is encoded in it. 
In the case of white noise, it is a constant function.

Tsirelson \cite{T2003-1} showed the following. 

\begin{theorem} For $r>0$, let $\sigma_r$ be a smooth positive even function with 
$\sigma_r(\lambda)= \log^{-r}|\lambda|$ for large $|\lambda|$. 
Then $\sigma_r$ is a spectral density function of an off-white noise, and the family 
$\{\sigma_r(\lambda)\}_{r>0}$ gives rise to mutually non-isomorphic type III product systems.
\end{theorem}

Tsirelson's construction has many faces. 
Bhat-Srinivasan \cite{BSr2005} systematically investigated the product systems arising from so-called sum systems 
by a purely functional analytic method, which recaptures Tsirelson's construction as a special case. 

Srinivasan and the author \cite{IS2008} showed that the $E_0$-semigroups corresponding to the product systems arising from sum systems 
are generalized CCR flows. 
Let $G$ be a real Hilbert space, and let $S=\{S_t\}_{t\geq 0}$ and $T=\{T_t\}_{t\geq 0}$ be $C_0$-semigroups acting on $G$ 
such that $T_t^*S_t=1_G$ and $S_t-T_t$ is a Hilbert-Schmidt operator for any $t\geq 0$. 
Then we can construct an $E_0$-semigroup $\alpha$ acting on $B(e^{G\otimes \C})$ by 
\begin{equation*}\alpha_t(W(f+ig))=W(S_tf+iT_tg),\quad f,g\in G.\end{equation*}
$E_0$-semigroups constructed in this way are called generalized CCR flows. 
Generalized CCR flows are either of type I or of type III (see \cite{BSr2005},\cite{I2009}). 
In the case of off-white noises, we can choose $G=L^2((0,\infty),\R)$ and $S$ to be the shift semigroup. 
The author \cite{I2007} gave the precise relationship between $T$ and the spectral density function $\sigma$ in this case. 
Using this correspondence, Srinivasan and the author \cite{IS2008} showed that there exists a continuous family of off-white noises 
whose spectral density functions converge to 1 at infinity, such that the family still gives mutually non-isomorphic type III product systems.

There is a certain similarity between the Toeplitz CAR flows and generalized CCR flows 
(for example, they are never of type II), and it is desirable to clarify their relationship. 
As a first step, we propose the following problem: 
determine the gauge groups of the Toeplitz CAR flows and generalized CCR flows.

Before ending this final section, we emphasize the importance of local von Neumann algebras associated with product systems 
in the results discussed in this section. 
Let $E$ be a product system. 
For $0 <s< t< 1$, we denote by $U_{s,t,1}$ the unitary map determined by 
\begin{equation*}U_{s,t,1}:E(s)\otimes E(t-s)\otimes E(1-t)\ni x\otimes y\otimes z\rightarrow xyz\in E(1).\end{equation*} 
For an interval $J=(s,t)\subset [0,1]$, 
we define the von Neumann algebra  $\cA^E(J)\subset B(E(1))$ associated with $J$ by 
\begin{equation*}\cA^E(J)=U_{s,t,1}(1_{E(s)}\otimes B(E(t-s))\otimes 1_{E(1-t)})U_{s,t,1}^*.\end{equation*}
We apply a similar definition in the case with $s=0$ or $t=1$. 
Then $\cA^E(J)$ is a type I subfactor of $B(E(1))$, and when two intervals $J_1$ and $J_2$ are disjoint, 
the corresponding algebras $\cA^E(J_1)$ and $\cA^E(J_2)$ commute with each other. 
For an open subset $O\subset [0,1]$, we define $\cA^E(O)$ to be the von Neumann algebra generated by $\cup_{J\subset O}\cA^E(J)$, 
which may not be of type I. 
The system of von Neumann algebras $\{\cA^E(J)\}_{J\subset [0,1]}$ is an important isomorphism invariant of the product system $E$, 
and it is an analogue of local observable algebras in algebraic quantum field theory.  
For example, Liebscher \cite{L2009} showed the following useful theorem, which Srinivasan and the author used in \cite{IS2010} 
to obtain a type I criterion. 

\begin{theorem} Let $E$ and $F$ be product systems. 
If there exists an isomorphism $\theta$ from $B(E(1))$ onto $B(F(1))$ satisfying $\theta(\cA^E((0,t)))=\cA^F((0,t))$ for all 
$0<t <1$, then $E$ and $F$ are isomorphic. 
\end{theorem}

The system of von Neumann algebras $\{\cA^E(I)\}_{I\subset [0,1]}$ and the isomorphism classes of von Neumann algebras $\cA^E(O)$    
are employed as isomorphism invariants of $E$ in \cite{IS2008},\cite{IS2010},\cite{T2003-1} to differentiate continuous families of 
product systems.


\section*{Acknowledgements}
The author would like to thank Alexis Alevras, Kenneth Davidson, Daniel Markiewicz, and 
Gilles Pisier for useful comments. 
This work is supported in part by the Grant-in-Aid for Scientific Research (B) 22340032, JSPS.


\newpage

\section{Appendix: Dilation theory and noncommutative Poisson boundary} 

The notion of noncommutative Poisson boundaries for normal ucp maps was introduced by the author \cite{I2002}. 
Let $N$ be a von Neumann algebra, and let $L$ be a weakly closed operator system in $N$, i.e. 
$L$ is a self-adjoint linear subspace of $N$ including the identity. 
It is known that if there exists a completely positive projection $E$ from $N$ onto $L$, 
then $L$ is a von Neumann algebra with respect to the Choi-Effros product $x\circ y=E(xy)$ 
(see \cite{CE77}). 

Let $P$ be a normal ucp map of $N$. 
We denote by $H^\infty(N,P)$ the fixed point set $\{x\in N;\; P(x)=x\}$ of $P$, 
whose members are called harmonic elements. 
Then $H^\infty(N,P)$ is a weakly closed operator system and it is the image of a completely positive projection from $N$. 
Indeed, we choose a free ultra-filter $\omega\in \beta\N\setminus \N$ and set 
\begin{equation*}E(x)=\mathrm{w}-\lim_{n\to\omega}\frac{1}{n}\sum_{k=0}^{n-1}P^k(x),\quad x\in N. \end{equation*}
Then $E$ is the desired projection. 
Although $E$ depends on the choice of $\omega$, the Choi-Effros product of $H^\infty(N,P)$ does not 
because an operator system may have at most one von Neumann algebra structure. 
Concrete realization of the von Neumann algebra structure of $H^\infty(N,P)$ 
is called the noncommutative Poisson boundary for $P$. 
This notion has proved to be particularly useful to capture structure that 
appears only after taking weak closure (see, for example, \cite{I2002}). 

When I invited Bill Arveson to Kyoto in 2004, he showed his interest in noncommutative Poisson boundaries. 
However, he did not seem to be happy about the fact that the Choi-Effros product is defined after the choice of the ultra-filter 
$\omega$ is made (even though it does not depend on $\omega$). 
When we met in 2007 at the Fields Institute, Bill told me in a brief conversation that the noncommutative Poisson boundary for $P$ 
is nothing but the fixed point algebra of the minimal dilation of $P$. 
It is so natural an idea. 
In fact, the usual measure theoretical construction of the Poisson boundary for a Markov operator 
in the commutative situation essentially uses the corresponding Markov process (see \cite{K92}), 
which is more or less the minimal dilation of the Markov operator. 
I felt a little embarrassed because such an idea had never occurred to me. 
I include a few consequences of Bill's observation here. 

Recall that the minimal dilation of $P$ consists of a von Neumann algebra $M$, a projection $p\in M$ whose central carrier is $1_M$, 
and a unital normal endomorphism $\alpha$ of $M$ such that $N=pMp$, $M$ is generated by $\cup_{n\geq 0}\alpha^n(N)$, and 
$P^n(a)=p\alpha^n(a) p$ for all $a\in N$ and $n\geq 1$. 
In this situation $\{\alpha^n(p)\}_{n=0}^\infty$ is an increasing sequence of projections whose limit is $1_M$. 
Without any modification, Theorem \ref{dilation} holds in the discrete time case too.

\begin{theorem} Let $M^\alpha=\{x\in M;\;\alpha(x)=x\}$ be the fixed point algebra of $\alpha$. 
Then the map $\theta :M^\alpha\ni x\mapsto pxp\in H^\infty(N,P)$ is a completely positive order isomorphism between the two operator systems. 
In particular, the von Neumann algebra $M^\alpha$ gives a concrete realization of the von Neumann algebra structure of $H^\infty(N,P)$. 
\end{theorem}

\begin{proof} It suffices to show that $\theta$ is a completely positive isometry of $M^\alpha$ onto $H^\infty(N,P)$ 
because the inverse of a unital completely positive isometry between two operator systems is automatically completely positive. 
Since the proof does not change after taking the tensor product with a matrix algebra, and 
$\theta$ is obviously completely positive, it suffices to show that $\theta$ is an isometry onto $H^\infty(N,P)$. 

For $x\in M^\alpha$, we have $P(pxp)=p\alpha(pxp)p=p\alpha(p)x\alpha(p)p$. 
Since $p\leq \alpha(p)$, we obtain $\theta(x)\in H^\infty(N,P)$. 
The map $\theta$ is obviously a contraction. 
Since the sequence $\{\alpha^n(p)\}_{n=1}^\infty$ converges to $1_M$ in the strong operator topology, the sequence 
$\{\alpha^n(\theta(x))\}_{n=1}^\infty$ converges to $x\in M^\alpha$ in the strong operator topology too, and 
so $\theta$ is an isometry. 
It remains to show that $\theta$ is a surjection. 
For $a\in H^\infty(N,P)$, we set 
\begin{equation*}a_n=\frac{1}{n}\sum_{k=0}^{n-1}\alpha^k(a),\end{equation*}
and we choose an accumulation point $x$ of the sequence $\{a_n\}_{n=1}^\infty$ in the weak operator topology. 
Then $x\in M^\alpha$. Since 
\begin{equation*}pa_np=\frac{1}{n}\sum_{k=0}^{n-1}p\alpha^k(a)p=\frac{1}{n}\sum_{k=0}^{n-1}P^k(a)=a,\end{equation*} 
we obtain $\theta(x)=a$. 
\end{proof}

The following statement is known in concrete examples where the martingale convergence theorem (commutative or noncommutative) 
is available (see \cite{I2004},\cite{KV83}). 
Thanks to the dilation theory, we are able to prove it in the general case. 

\begin{corollary} For any $a,b\in H^\infty(N,P)$, the sequence $\{P^n(ab)\}_{n=1}^\infty$ converges to 
the Choi-Effros product $a\circ b$ in the strong operator topology. 
\end{corollary}

\begin{proof} Let $x=\theta^{-1}(a)$, $y=\theta^{-1}(b)$. 
Then 
\begin{equation*}P^n(ab)=p\alpha^n(pxpyp)p=p\alpha^n(p)x\alpha^n(p)y\alpha^n(p)p=px\alpha^n(p)yp,\end{equation*}
which converges to $pxyp=\theta(xy)$ in the strong operator topology. 
\end{proof} 

The following example was suggested by Bill. 
Let $H^2$ be the Hardy space of the unit disk, and let $p$ be the projection from $L^2(\T)=\ell^2(\Z)$ onto $H^2$. 
Let $v\in B(H^2)$ be the unilateral shift, and let $u\in B(L^2(\T))$ be the bilateral shift. 
We introduce a normal ucp map $P$ of $N=B(H^2)$ by $P(a)=v^*av$. 
Then the space of harmonic elements $H^\infty(N,P)$ consists of the Toeplitz operators $T_f$, $f\in L^\infty(\T)$ 
(see \cite[Theorem 4.2.4]{SCST}). 
The minimal dilation of $P$ is given by $(M=B(L^2(\T)), p, \alpha=\Ad u^*)$, and 
the fixed point algebra $M^\alpha$ is $L^\infty(\T)$. 
In summary, the map $\theta:L^\infty(\T)\ni f\mapsto T_f\in H^\infty(N,P)$ is a completely positive order isomorphism 
between the two operator systems, and $\{{v^*}^nT_fT_gv^n\}_{n=1}^\infty$ converges to $T_{fg}$ 
in the strong operator topology for any $f,g\in L^\infty(\T)$.  

We end this note with an example coming from random walks on discrete groups discussed in \cite{I2004}. 
Let $G$ be a discrete group, and let $\mu$ be a probability measure on $G$ whose support $S$ generates $G$ as a semigroup. 
We define a ucp map $\cP_\mu$ acting on $\ell^\infty(G)$ by the right convolution operator $\cP_\mu(f)=f*\check{\mu}$ where 
$\check{\mu}(g)=\mu(g^{-1})$. 
Then $\cP_\mu$ gives rise to a random walk on $G$ with the transition probability $\mathrm{p}(g,h):=\mathrm{Pr}(X_{n+1}=h|X_n=g)$ 
given by $\mu(g^{-1}h)$. 
We denote by $H^\infty(G,\mu)$ the space of bounded harmonic functions for $\cP_\mu$ and by $(\partial G,\nu)$ the 
Poisson boundary with the harmonic measure for $\cP_\mu$ (see \cite{K92},\cite{KV83}) . 

As in \cite{I2004}, we extend $\cP_\mu$ to $N=B(\ell^2(G))$ by 
\begin{equation*}\cQ_\mu(a)=\sum_{g\in G}\mu(g)\rho_ga\rho_g^{-1},\end{equation*}
where $\rho$ is the right regular representation. 
Then the noncommutative Poisson boundary for $\cQ_\mu$ is the boundary crossed product $L^\infty(\partial G,\nu)\rtimes G$ 
(see \cite{I2004},\cite{JN2007}). 
Let $\lambda$ be the left regular representation of $G$. 
Then $H^\infty(B(\ell^2(G)),\cQ_\mu)$ is spanned by $H^\infty(G,\mu)\lambda_G$.

In the rest, we identify the minimal dilation $(M,p,\alpha)$ of $P=\cQ_\mu$, and give a new description of 
the boundary crossed product.  
Let $\theta:M^\alpha \ni x\mapsto pxp\in H^\infty(N,P)$ be as before. 
Since $\lambda_g$ is a unitary in the multiplicative domain of $P$, it is easy to show the following lemma.  

\begin{lemma}\label{unitary} Let $u_g=\theta^{-1}(\lambda_g)$. Then $\{u_g\}_{g\in G}$ is a unitary representation of $G$ in $M^\alpha$ 
commuting with $p$. 
\end{lemma}

We may assume that $M=B(H)$ and $p$ is a projection onto a closed subspace $H_0$ of $H$ identified with $\ell^2(G)$. 
By minimality, the Hilbert space $H$ is spanned by 
\begin{equation*}\alpha^n(e_{g_n,h_n})\alpha^{n-1}(e_{g_{n-1},h_{n-1}})\cdots \alpha(e_{g_1,h_1})\delta_{g_0},\end{equation*} 
where $\{e_{g,h}\}_{g,h\in G}$ is the system of matrix units in $B(\ell^2(G))$ with respect to the orthonormal basis 
$\{\delta_g\}_{g\in G}$. 
Thanks to Lemma \ref{unitary} and $e_{g,h}=\delta_g\lambda_{gh^{-1}}$, we have $\alpha^n(e_{g,h})=\alpha^n(\delta_g)u_{gh^{-1}}$. 
Thus we see that $H$ is spanned by 
\begin{equation*}\zeta^n(g_0,g_1,\ldots,g_n)=\alpha^n(\delta_{g_n})\alpha^{n-1}(\delta_{g_{n-1}})\cdots \alpha(\delta_{g_1})\delta_{g_0}.\end{equation*} 
Direct computation shows
\begin{align*}
\lefteqn{\inpr{\zeta^n(g_0,g_1,\ldots,g_n)}{\zeta^n(h_0,h_1,\ldots,h_n)}} \\
&=\delta_{g_n,h_n}
\inpr{\alpha^{n-1}(p\alpha(\delta_{g_n})p)\zeta^{n-1}(g_0,g_1,\ldots,g_{n-1})}{\zeta^{n-1}(h_0,h_1,\ldots,h_{n-1})}\\
&=\delta_{g_n,h_n}\inpr{\alpha^{n-1}(\delta_{h_{n-1}}P(\delta_{g_n})\delta_{g_{n-1}})
\zeta^{n-2}(g_0,g_1,\ldots,g_{n-2})}{\zeta^{n-2}(h_0,h_1,\ldots,h_{n-2})}\\
&=\delta_{g_n,h_n}\delta_{g_{n-1},h_{n-1}}\mathrm{p}(g_{n-1},g_n)\\
&\times \inpr{\alpha^{n-1}(\delta_{g_{n-1}})\zeta^{n-2}(g_0,g_1,\ldots,g_{n-2})}{\zeta^{n-2}(h_0,h_1,\ldots,h_{n-2})},
\end{align*}
and therefore 
\begin{align*}
\lefteqn{\inpr{\zeta^n(g_0,g_1,\ldots,g_n)}{\zeta^n(h_0,h_1,\ldots,h_n)}} \\
 &=\delta_{g_0,h_0}\delta_{g_1,h_1}\cdots \delta_{g_n,h_n}\mathrm{p}(g_0,g_1)\mathrm{p}(g_1,g_2)\cdots \mathrm{p}(g_{n-1},g_n).
\end{align*}
We also have 
\begin{equation*}\zeta^n(g_0,g_1,\ldots,g_n)=\sum_{g\in G}\zeta^{n+1}(g_0,g_1,\ldots,g_n,g).\end{equation*}
Now we can see that $H$ is identified with the $L^2$-space over the path space $G^\infty$. 
Let $C^n(g_0,g_1,\ldots,g_n)$ be the cylinder set 
\begin{equation*}C^n(g_0,g_1,\ldots,g_n)=\{(x_n)_{n=0}^\infty\in G^\infty;\; x_i=g_i,\; 0\leq i\leq n\},\end{equation*}
and let $m$ be the measure on $G^\infty$ determined by $m(C^0(g))=1$ and 
\begin{equation*}m(C^n(g_0,g_1,\ldots,g_n))=\mathrm{p}(g_0,g_1)\mathrm{p}(g_1,g_2)\cdots \mathrm{p}(g_{n-1},g_n).\end{equation*}
Then $L^2(G^\infty,m)\ni 1_{C^{n}(g_0,g_1,\ldots,g_n)}\mapsto \zeta^n(g_0,g_1,\ldots,g_n)\in H$ 
gives a unitary operator, and we identify the two Hilbert spaces in what follows. 
Note that $\delta_g\in H_0$ is identified with $1_{C^0(g)}$. 
Since 
\begin{equation*}u_g\zeta^n(g_0,g_1,\ldots,g_n)=\zeta^n(gg_0,gg_1,\ldots,gg_n),\end{equation*} 
we have $u_g\xi((x_n))=\xi((g^{-1}x_n))$ for all $\xi\in L^2(G^\infty,m)$. 

Let $H_n$ be the closed subspace of $H$ spanned by $\{1_{C^n(g_0,g_1,\ldots,g_n)}\}_{g_0,g_1,\ldots,g_n\in G}$, 
and let $p_n$ be the projection onto $H_n$. 
Then we have $p=p_0$. 
Let $j_k:\ell^\infty(G)\rightarrow B(H)$ be the representation of $\ell^\infty(G)$ defined by 
$j_k(f)\xi((x_n))=f(x_k)\xi((x_n))$. 
We identify $f\in \ell^\infty(G)$ and $\lambda_g$ in $N$ with $j_0(f)p$ and $u_gp$ in $M$. 
Then we have $\alpha^k(j_0(f)p)=j_k(f)p_k$. 
Since $M$ is generated by $\cup_{k\geq 0} \alpha^k(N)$, and $N$ is spanned by $\ell^\infty(G)\lambda_G$, 
the homomorphism $\alpha$ is determined by the condition 
$\alpha(u_g)=u_g$ and $\alpha(j_k(f)p_k)=j_{k+1}(f)p_{k+1}$. 

Since $\alpha$ is a unital homomorphism of $M=B(H)$, it is implemented by a Cuntz algebra representation. 
Let $T:G^\infty\rightarrow G^\infty$ be the time shift $(Tx)_n=x_{n+1}$. 
For each $g\in S$, we set 
\begin{equation*}S_g\xi((x_n))=\frac{\delta_{x_0g,x_1}}{\sqrt{\mu(g)}}\xi\circ T((x_n)),\quad \xi\in L^2(G^\infty,m).\end{equation*}
Then $S_g\in B(H)$ is an isometry with the adjoint operator given by 
\begin{equation*}S_g^*\xi((x_n))=\sqrt{\mu(g)}\xi(x_0g^{-1},x_0,x_1,\cdots),\quad \xi\in L^2(G^\infty,m),\end{equation*}
and so we have $S_gS_g^*\xi((x_n))=\delta_{x_0g,x_1}\xi((x_n))$. 
The range projections $\{S_gS_g^*\}_{g\in S}$ are mutually orthogonal, and the summation converges to $1_M$. 
Thus $\{S_g\}_{g\in S}$ satisfy the Cuntz algebra $\cO_n$ relation with $n=\# S$. 
Let  
\begin{equation*}\beta(x)=\sum_{g\in S}S_gxS_g^*,\quad x\in M.\end{equation*}
Then it is easy to see $\beta(u_g)=u_g$, $\beta(\eta)=\eta\circ T$ for $\eta\in L^\infty(G^\infty,m)$, and 
$\beta(p_k)=p_{k+1}$. 
Thus we get $\alpha=\beta$. 

The above argument shows that the fixed point algebra $M^\alpha$ is the commutant of the Cuntz algebra 
$\cO_n=C^*\{S_g\}_{g\in S}$. 
We specify a state of $\cO_n$ giving this representation. 
We claim that $1_{C^0(e)}$ is a separating vector for $M^\alpha$. 
Recall that $H^\infty(N,P)$ is spanned by $H^\infty(G,\mu)\lambda_G$. 
For $f\in H^\infty(G,\mu)$, we have 
\begin{align*}\lefteqn{
\inpr{\theta^{-1}(f\lambda_g)1_{C^0(e)}}{1_{C^0(e)}}
=\lim_{k\to\infty}\inpr{\alpha^k(j_0(f)p)u_g1_{C^0(e)}}{1_{C^0(e)}}
}\\
&=\lim_{k\to\infty}\inpr{j_k(f)p_k1_{C^0(g)}}{1_{C^0(e)}}=\delta_{g,0} \lim_{k\to\infty}\inpr{j_k(f)1_{C^0(e)}}{1_{C^0(e)}}\\
 &=\delta_{g,0} \lim_{k\to\infty}\sum_{h\in G}\mathrm{p}^{(k)}(e,h)f(h)
 =\delta_{g,0}f(e),
\end{align*}
where $\mathrm{p}^{(k)}(e,g)$ is the $k$-step transition probability. 
This shows that $1_{C^0(e)}$ induces a faithful normal state of the boundary crossed product 
$L^\infty(\partial G,\nu)\rtimes G$ corresponding to the harmonic measure $\nu$ (see \cite{I2004}), 
and so the claim is proved. 
In consequence $1_{C^0(e)}$ is cyclic for $\cO_n=C^*\{S_g\}_{g\in S}$. 
We denote by $\omega_\mu$ the state of $\cO_n$ given by $1_{C^0(e)}$. 
Note that we have $S_g^*1_{C^{0}(h)}=\sqrt{\mu(g)}1_{C^0(hg)}$. 

\begin{corollary} There exists a state $\omega_\mu$ of the Cuntz algebra $\cO_n=C^*\{S_g\}_{g\in S}$ 
given by 
\begin{equation*}\omega_\mu(S_{g_1}S_{g_2}\cdots S_{g_k}S_{h_l}^*S_{h_{l-1}}^*\cdots S_{h_1}^*)
=\prod_{i=1}^k\mu(g_i)^{1/2}\prod_{j=1}^l\mu(h_j)^{1/2}\delta_{g_1g_2\cdots g_k,h_1h_2\cdots h_l},\end{equation*}
and the boundary crossed product $L^\infty(\partial G,\nu)\rtimes G$ is isomorphic to $\pi_{\omega_\mu}(\cO_n)'$, 
where $\pi_{\omega_\mu}$ is the GNS representation of $\omega_\mu$. 
\end{corollary}

\def\refname{References for Appendix}

\end{document}